\documentclass[12pt]{amsart}
\usepackage{latexsym}
\usepackage{amsfonts}
\usepackage{graphicx}
\usepackage{amsfonts}
\usepackage{amsmath,amscd}
\usepackage{color}
\usepackage[colorlinks=true, pdfstartview=FitV, linkcolor=blue, citecolor=blue, urlcolor=blue]{hyperref}
\usepackage{epsfig}
\usepackage{amsthm,amsfonts}
\usepackage{amssymb,graphicx,color}
\usepackage[all]{xy} 
\usepackage{verbatim}
\usepackage{hyperref}

\def\ov#1{\overline{#1}}
\def\ti#1{\tilde{#1}}

\newcommand{\field}[1]{\mathbb{#1}}
\newcommand{\C}{\field{C}}

\newcommand{\PP}{\field{P}}

\DeclareMathOperator{\Sing}{Sing}
\DeclareMathOperator{\Reg}{Reg}
\DeclareMathOperator{\GSDS}{GSDS}
\DeclareMathOperator{\SDS}{SDS}
\DeclareMathOperator{\Lin}{Lin}

\newtheorem{definition}{Definition}[section]

\newtheorem{lemma}[definition]{Lemma}
\newtheorem{theorem}[definition]{Theorem}
\newtheorem{corollary}[definition]{Corollary}
\newtheorem{proposition}[definition]{Proposition}

\newtheorem{remark}[definition]{Remark}

\makeatletter
\@namedef{subjclassname@2020}{%
  \textup{2020} Mathematics Subject Classification}
\makeatother

\title [Generic symmetry defect set of an algebraic  curve]{Generic symmetry defect set of an algebraic curve} \makeatletter

\@addtoreset{equation}{section}
\makeatother

\thanks{}
\date {\today}

\author{L.R.G. Dias \& M. Farnik \& Z. Jelonek}

\address[L.R.G. Dias] {Faculdade de Matem\'atica, Universidade Federal de Uberl\^andia, Av. Jo\~ao Naves de \'Avila 2121, 1F-153 - CEP: 38408-100, Uberl\^andia, Brasil}
\email{lrgdias@ufu.br}

\address[M. Farnik]{Jagiellonian University\\
Faculty of Mathematics and Computer Science\\
{\L}ojasie\-wi\-cza~6, 30-348 Krak\'ow, Poland}
\email{michal.farnik@gmail.com}

\address[Z. Jelonek]{Instytut Matematyczny\\
	Polska Akademia Nauk\\
	\'Sniadeckich 8, 00-956 Warszawa, Poland}
\email{najelone@cyf-kr.edu.pl}

\keywords{symmetric defect, bifurcation points, Wigner caustic.}

\subjclass[2020]{14 D 06, 14 Q 20}

\begin{document}

\begin{abstract} Let $X \subset \C^{2n}$ be an  $n$-dimensional algebraic variety. We define the algebraic version of the generic symmetry defect set (Wigner caustic) of $X$. Moreover, we compute its singularities for $X_d$ being a generic curve of degree $d$  in~$\C^2$.
\end{abstract}

\maketitle

\section{Introduction}

Over the last two decades numerous methods have been developed to
study affine geometry of surfaces and curves, especially their
affinely invariant symmetry characteristics. The symmetry sets
\cite{Bruce, Damon, Damon1} and the center symmetry sets were
investigated extensively in \cite{Jan, GH, GJ, DP}. Several constructions of
the set equivalent to the point of central symmetry for perturbed
centrally symmetric ovals were presented in the literature and
resulted in the kind of symmetry defect called center symmetry set.
The center symmetry set directly appears in the construction of the
so-called Wigner caustic. This caustic is obtained by the stationary
phase method applied to the semiclassical Wigner function which
completely describes a quantum state in the symplectic phase space
\cite{berry}. It is {built} of points where the central symmetry, i.e.
the number of intervals ending in the surface and passing centrally
through that point, changes. We call this set a symmetry defect or
bifurcation set. In \cite{jjr}   this construction was generalized for algebraic
varieties $Z^n\subset \C ^{2n}$.

This note  is motivated by the recent results of \cite{jjr}. We address the problem of how to introduce a {\it  generic symmetry defect set} ($\GSDS$) associated to $n$-dimensional variety $X\subset \C^{2n}$. Note that in the algebraic case we cannot use a general deformation of $X$ as it is done in the smooth case. To avoid this difficulty we will consider the linear deformation of a pair $(X,X)$ and we will study the symmetry defect set of this pair. We will show that the generic symmetry defect set is an irreducible algebraic hypersurface with nice singularities, which is defined up to the ambient homeomorphism of $\C^{2n}$.

We also show:

\vspace{2mm}

\noindent {\bf Theorem \ref{th_GSDSXY}} {\it Let $X,Y\subset \C^2$ be generic plane curves of degree $d_1$ and $d_2$, respectively, with $d_1,d_2\geq 2$. Then $C'=\GSDS(X,Y)$ is an irreducible curve with 
$$c=12{d_1 \choose 2}{d_2 \choose 2}$$
cusps and 
$$n=2{d_1 \choose 2}{d_2 \choose 2}[(d_1+d_2)^2-d_1-d_2-10]$$
nodes.
It has degree, genus and Euler characteristic equal respectively to
\begin{align*}
\deg(C')&=d_1d_2(d_1+d_2-2),\\
g(C')&=d_1d_2(2d_1d_2-3(d_1+d_2)+4)+1,\\
\chi(C')&=-d_1d_2(4d_1d_2-5(d_1+d_2)+6).
\end{align*}
}

This yields as a consequence

\noindent {\bf Corollary \ref{cor_GSDSX}} {\it Let $X\subset \C^2$ be a generic plane curve of degree $d\geq 2$. Then $\GSDS(X)$ is an irreducible curve of degree $2d^2(d-1)$ and genus $g=2d^2(d^2-3d+2)+1$ with $c=12{d \choose 2}^2$
cusps and $n=4{d \choose 2}^2[2d^2-d-5]$ nodes. Moreover, $\chi(\GSDS(X))=-d^2(4d^2-10d+6)$.
}

\section{Symmetry defect set}

Let $X^n\subset \C^{2n}$ be a smooth manifold of dimension $n$. For a given point $a\in \C^{2n}$ we are interested in the number $\mu(a)$ of pairs of points $x,y\in X$ such that $a$ is the center of the interval $\overline{xy}$, $a=\frac{x+y}{2}$ (mid point map).
We have showed in \cite{jjr} that if $X$ is an algebraic manifold in a general
position, then there is a closed set $B\subset \C^{2n},$ such that the function $\Phi\colon X\times X\ni (x,y)\mapsto (x+y)/2\in \C^{2n}$ is a differentiable covering outside $B$. We call the minimal such a set $B=B(X)$ the \emph{symmetry defect set} of $X$ and denote it by $\SDS(X)$.
We have showed that the symmetry defect set is an algebraic hypersurface and
consequently the function $\mu$ is constant and positive outside $\SDS(X)$. We have estimated the number $\mu$ and the degree of the hypersurface $\SDS(X).$ 

In fact we can do this in a more general setting, we can start from a pair  $X^r, Y^s\subset \C^n$ of smooth manifolds of dimensions $r$ and $s$
respectively, where $r+s=n$. Again, if $X, Y$  are algebraic manifolds in a general position, then there is a closed  set $B\subset
\C^n$ such that the function $\Phi\colon X\times Y\ni (x,y)\mapsto (x+y)/2\in \C^n$ is a differentiable covering outside $B$. The minimal such a set is an algebraic hypersurface which we call the \emph{symmetry defect set} of $X$ and $Y$ and denote it by $\SDS(X,Y)$.

However, in general the hypersurface $\SDS(X)$ (or $\SDS(X,Y)$) has bad singularities. 
In this paper we will introduce the stable (generic) versions of $\SDS(X)$ and $\SDS(X,Y)$. To do this we cannot use general deformation of $X$ as in the smooth case. In fact we can use here only linear deformations, but this is not enough to obtain a stable $\SDS$. The good solution is to consider the pair $(X,X)$ and to deform this pair. 

\begin{definition}
Let $M$ be a smooth manifold. We will say that the hypersurface $X\subset M$ has Thom-Boardman singularities, if there is a smooth manifold $N$ and a Thom-Boardman mapping $F\colon N\to M$ such that  $X$ is a discriminant of $F$.
\end{definition}

We have:

\begin{theorem}
Let $X^r,Y^s\subset\C^{2n}$ (where $r+s=2n$) be smooth algebraic manifolds and let $\Lin(2n,2n)$ denote the group of affine linear automorphisms of $\C^{2n}$. There is a Zariski open subset $U\subset \Lin(2n,2n)\times \Lin(2n,2n)$ such that for every $G,H\in U$ the hypersurface $\SDS(G(X),H(Y))$ has only Thom-Boardman singularities. Moreover, if $(G,H)\in U$ and $(G_1,H_1)\in U$ then the hypersurfaces $\SDS(G(X),H(Y))$  and $\SDS(G_1(X),H_1(Y))$ are ambient homeomorphic.
\end{theorem}

\begin{proof}
Let us consider the variety $\Gamma:=X\times Y\subset \C^{2n}\times \C^{2n}.$ Consider the general projection $\pi: \Gamma\to\C^{2n}.$
By Mather's theorem (see \cite{math} and \cite{fjr}) there is an Zariski open subset $V\subset \Lin(4n,2n)$ such that if $\pi\in V$ then the projection 
$\pi_{|\Gamma}$ is transversal to the Thom-Boardman strata. Moreover by \cite{jel} we can shrink $V$ such that all projections $\pi_{|\Gamma}$ are topologically equivalent. Now consider the map $\Psi\colon \Lin(2n,2n)\times \Lin(2n,2n)\ni (G,H)\mapsto \frac{G+H}{2}\in \Lin(4n,2n)$ (here $\frac{G+H}{2}(x,y)=\frac{G(x)+H(y)}{2}$). Now it is enough to take $U=\Psi^{-1}(V).$
\end{proof}

\begin{definition}
Let $X^n,Y^n\subset\C^{2n}$ be smooth algebraic manifolds. By the generic symmetry defect set we mean the set $\SDS(G(X),H(Y))$ where 
$G,H\in \Lin(2n,2n)$ are sufficiently general. We denote it by $\GSDS(X,Y)$. This set is defined up to the ambient homeomorphism. We write $\GSDS(X)=\GSDS(X,X)$.
\end{definition}

\begin{remark}
{\rm Of course we can take here $G$ to be identity and take the map $H$ as close to the identity as we wish.}
\end{remark}

We show in Section \ref{sec_nieroz} that if neither $X$ nor $Y$ is a linear space then $\GSDS(X,Y)$ is an irreducible hypersurface with Thom-Boardman singularities. In particular in the case $n=2$ the curve $\GSDS(X,Y)$ has only cusps and nodes as singularities and the number of these cusps and nodes is an affine invariant of $X$ and $Y$. In Section~\ref{sec_GSDS} we compute these numbers for generic plane curves $X_{d_1}$ and $Y_{d_2}$ of degrees $d_1$ and $d_2$, respectively.

\section{The set \texorpdfstring{$\GSDS(X,Y)$}{GSDS(\textit{X,Y})} is irreducible} \label{sec_nieroz}

We need the following  theorem concerning the properties of linear systems on algebraic varieties, due to E. Bertini \cite{aki}:

\begin{theorem}\label{bertini}
Let $V$ be an algebraic variety over an algebraically closed field $k$ of characteristic $0$, let $L$ be a linear system without fixed components 
on $V$ and let $W$ be the image of the variety $V$ under the mapping $j_L$  given by $L$. Then:

(1) If dim $W>1$ then almost all the divisors of the linear system $L$ (i.e. all except a closed proper subset in the parameter space $\PP(L)$)
are irreducible reduced algebraic varieties.

(2) Almost all divisors of $L$
have no singular points outside the basis points of the linear system $L$ and the singular points of the variety $V$.
\end{theorem}

Let us recall that if we have a projective hypersurface $X\subset \PP^n$ then we have the dual map $\phi\colon \Reg(X)\ni x\mapsto T_x X\in {\PP^n}^*$ and the variety $X^*=cl(\phi(X))$ is called the dual to $X$. More generally if $X$ is of arbitrary dimension and $x\in \Reg(X)$ then the hyperplane $H\in {\PP^n}^*$ is tangent to $X$ if  $T_x X$ is  contained in $H$ (here we consider $T_x X$ as a projective subspace of $\PP^n$). We have a well-known reflexivity property: $X^{**}=X$. 
From this it easily follows that:

\begin{proposition}
Let $X$ be a hypersurface in $\PP^n.$ If $\dim X^*=s <n-1$ then $X$ is $n-1-s$ ruled, i.e., through every point $x\in X$ there is a projective linear subspace $L^{n-1-s}\subset X$ of dimension $n-1-s$.
\end{proposition}

Now we can pass to the main result of this section:

\begin{theorem}\label{glowne1}
Let $X\subset\C^n$ be a smooth algebraic variety of dimension $k$. Assume that $X$ is not $k-1$ ruled. Then the critical set $C_\pi$ and the discriminant $\Sigma_\pi=\pi(C_\pi)$ of a generic projection $\pi\colon X\to \C^k$ are irreducible. Moreover, $C_\pi$ is smooth and $\Sigma_\pi$ has only Thom-Boardman singularities.
\end{theorem}

\begin{proof}
We can write $\pi=\rho\circ\phi$ where $\phi\colon X\to\C^{k+1}$ and $\rho\colon \phi(X)\to\C^k$ are generic projection. By Mather's Theorem (see \cite{math}) the variety $\phi(X)$ has only normal crossings outside critical values of $\phi$. Since the set of critical values has codimension two or more we can assume that $\phi(X)$ has only normal crossings. Let $S_1, \ldots, S_m$ be all irreducible components of $\Sing(\phi(X))$.
Take sufficiently general points $a_i\in S_i$. 
Then two different branches of $\phi(X)$ meet in $a_i$, in particular we have two different tangent spaces $R_i$ and $P_i$ to these branches. The set of projections $\rho\colon \phi(X)\to\C^k$ which induce isomorphisms on all spaces $R_i, P_i$ is open and dense in the family of such projections $\rho\colon \phi(X)\to\C^k$. 

Hence in fact we may prove our theorem only for $Y=\Reg(\phi(X))$ and a generic projection $\rho\colon Y\to \C^k$. We apply here Theorem \ref{bertini}. Note that the critical set of a projection $\pi$ given by equations $(\sum a_{1i}x_i,\ldots,\sum a_{ki} x_i)$ is described by an equation:
\begin{equation}\label{eq_system}
 \det \left[\begin{matrix} f_{x_1} & \ldots & f_{x_{k+1}} \\ a_{11} & \ldots & a_{1,k+1}\\  \ldots \\  a_{k,1} & \ldots & a_{k,k+1} \end{matrix}\right]=0.
\end{equation}
where $f=0$ is the equation of $\phi(X)$ and $f_{x_i}:=\frac{\partial f}{\partial x_i}.$ 

Note that if $\rho$ goes through all possible projections then (\ref{eq_system}) forms a linear system $L$ on $Y$. Of course it has no base points on $Y$. It is enough to prove that $\dim j_L(Y)>1$. First note that if $x,y\in Y$ and $T_x$ is not parallel to $T_y$ then there is a section $s$ of $L$ such that $s(x)=0$ and $s(y)\not=0$. Indeed, it is enough to take a  projection $\rho$ with center $P$ such that $P\in \overline{T_x X}$ but $P\not\in \overline{T_y Y}$ and $s$ the section given by $\rho$.  Hence $L$ separates points with not parallel tangent spaces. Now we need the following:

\begin{lemma}\label{lemat}
Let $S\subset Y$ be an irreducible subvariety such that for all $y\in S$ the tangent spaces $T_y Y$ are parallel. Then there is a hyperplane $H$ such that $S\subset H$ and $T_y Y=H$ for every $y\in S$.
\end{lemma}

\begin{proof}
We may assume that $S$ is not a point. Assume that all tangent spaces $T_y Y$, $s\in S$ are parallel to some hyperplane $W$ given by equation $h=0$. If $h(Y)=c$ then for $H:=\{h=c\}$ we have $Y\subset H$ and all tangent spaces are equal to $H$. If $h(Y)=\C$, then by Sard's theorem there is a $c\in \C$ such that the hyperplane $H$ is transversal to $\Reg(S)$, a contradiction.  
\end{proof}

Now assume that $S$ is a component of a fiber of the mapping $j_L$. Then all tangent planes $T_y Y$, $y\in S$ are parallel. From Lemma \ref{lemat} we have that $T_y Y=H$ for some hyperplane $H$ and every $y\in S$. In particular $Y$ is in the fiber of the mapping $\psi\colon Y\ni y\mapsto \overline{T_y Y}\in Y^*$. Since $X$ is not $k-1$ ruled and $\phi(X)$ is a linear birational projection of $X$, the variety $\overline{\phi(X)}$
is also not $k-1$ ruled. In particular general fibers of the mapping $\psi$ have dimension less than $k-1$. Hence dim $j_L(Y)>1$.

The last statements follows from Mather's projection theorem.
\end{proof}

\begin{remark}
{\rm The assumption that $X$ is not $k-1$ ruled is essential. Indeed, every generic projection of a cylinder $D=\{ (x,y,z)\in \C^3: x^2+y^2=1\}$
to $\C^2$ has reducible critical set and reducible discriminant.}
\end{remark}

\begin{corollary}
Let $X^n,Y^n\subset\C^{2n}$ be a smooth algebraic varieties. If neither $X$ nor $Y$ is linear then the set $\GSDS(X,Y)$ is an irreducible hypersurface.
\end{corollary}

\begin{proof} If $X$ and $Y$ are not linear then they are ruled in dimension at most $n-1$. Hence $X\times Y$ is ruled in dimension at most $2n-2$ and we can apply Theorem~\ref{glowne1}.
\end{proof}

\section{\texorpdfstring{$\GSDS(X,Y)$}{GSDS(\textit{X,Y})} for generic plane curves \texorpdfstring{$X$}{\textit{X}}, \texorpdfstring{$Y$}{\textit{Y}}} \label{sec_GSDS}

\begin{theorem}\label{th_GSDSXY}
Let $X,Y\subset \C^2$ be generic plane curves of degree $d_1$ and $d_2$, respectively, with $d_1,d_2\geq 2$. Then $C'=\GSDS(X,Y)$ is an irreducible curve with 
$$c=12{d_1 \choose 2}{d_2 \choose 2}$$
cusps and 
$$n=2{d_1 \choose 2}{d_2 \choose 2}[(d_1+d_2)^2-d_1-d_2-10]$$
nodes.
It has degree, genus and Euler characteristic equal respectively to
\begin{align*}
\deg(C')&=d_1d_2(d_1+d_2-2),\\
g(C')&=d_1d_2(2d_1d_2-3(d_1+d_2)+4)+1,\\
\chi(C')&=-d_1d_2(4d_1d_2-5(d_1+d_2)+6).
\end{align*}
\end{theorem}

\begin{proof}
To make the proof easier to read we will divide it into five steps.

{\bf Step 1:} Basic definitions and construction of $C'=\GSDS(X,Y)$.

Let $X=\{(x,y)\in\C^2:\ f(x,y)=0\}$ and $Y=\{(z,w)\in\C^2:\ g(z,w)=0\}$ where $f(x,y)$ and $g(z,w)$ are a general polynomials of degree $d_1$ and $d_2$, respectively. We will consider $\Gamma=X\times Y\subset \C^4\subset \PP^4$. We denote the coordinates in $\C^4$ by $x,y,z,w$ and the coordinates in $\PP^4$ by $x,y,z,w,t$ (slightly abusing notation). We denote the partial derivative by a lower index, e.g. $f_x$. By $\overline{f}$ we denote the homogenization of $f$ with respect to $t$ and by $\tilde{f}$ we denote the dehomogenization of $\overline{f}$ with respect to $x$. Note that in $\tilde{f}_x$ we first take the partial derivative and then the homogenization and dehomogenization. We will also denote the projective closure of a variety $V$ by~$\overline{V}$.

Since $X$ is generic and has degree $d_1$, there are $d_1$ distinct points of intersection of $\overline{X}$ with the line at infinity, say $(a_i:b_i:0)$, for $1\leq i\leq d_1$. Similarly, let $(c_i:d_i:0)$, $1\leq i\leq d_2$, be the points of intersection of $\overline{Y}$ with the line at infinity. Let $P_i=(a_i:b_i:0:0:0)$ and $Q_i=(0:0:c_i:d_i:0)$ be the corresponding points in $\overline{\Gamma}$. Note that the $d_1d_2$ lines spanned by $P_i$ and $Q_j$ also lie in $\overline{\Gamma}$. Since $\Gamma$ is a surface of degree $d_1d_2$, the intersection of $\overline{\Gamma}$ with the hyperplane at infinity consists only of those $d_1d_2$ lines.

Now consider a generic projection $\pi\colon \Gamma\to\C^2.$ Composing with a linear change of coordinates in the target $\C^2$ we can assume that $\pi(x,y,z,w)=(x+az+bw,y+cz+dw)$, where $a,b,c,d$ are generic. Let
\begin{equation}
h(x,y,z,w)=\det \left[\begin{matrix} f_{x}(x,y) & f_y(x,y) & 0 &0\\ 0 & 0 & g_{z}(z,w) & g_w(z,w) \\ 1 & 0 & a & b\\0 & 1 & c & d  \\   \end{matrix}\right].
\end{equation}

Let $C$ denote the critical set of $\pi$. It is given by equations:
\begin{equation}\label{eq_C}
f=0,\ g=0,\ h=bf_xg_z-af_xg_w+df_yg_z-cf_yg_w=0.
\end{equation}

We denote by $C'$ the curve $\GSDS(X,Y)=\pi(C)$.

Since $X$ is generic, the values $\ov{f}_x(P_i)$, $\ov{f}_y(P_i)$, $\ov{g}_z(Q_i)$ and $\ov{g}_w(Q_i)$ are nonzero. Thus, for a generic choice of $a,b,c,d$ we obtain
$$\ov{h}(\alpha a_i:\alpha b_i:\beta c_j:\beta d_j:0)=$$
$$\alpha\beta\left[\ov{f}_x(P_i)(b\ov{g}_z-a\ov{g}_w)(Q_j)+
\ov{f}_y(P_i)(d\ov{g}_z-c\ov{g}_w)(Q_j)\right]=\alpha^{d_1-1}\beta^{d_2-1} A_{i,j}$$
for some nonzero constants $A_{i,j}$. This means that the hypersurface defined by $\ov{h}$ intersects 
the lines spanned by $P_i$ and $Q_j$ only at the points $P_i$ and $Q_j$. It follows that $\ov{C}$ is a complete intersection given by $\ov{f}=\ov{g}=\ov{h}=0$, at least set-theoretically.

{\bf Step 2:} Examining the branches of $\ov{C}$ at infinity.

Now we will focus on examining the branches of $\ov{C}$ at infinity. Without loss of generality we may examine only branches through $P_1$ and assume that $P_1=(1:b_1:0:0:0)$. We will work in the neighborhood $U\cong\C^4$ defined in $\PP^4$ by $x\neq 0$. The notation $o(t^k)$ means here a series in $t$ of order strictly larger than $k$.

We can write down $\ti{f}$ up to multiplying by a constant as
$$\ti{f}(y,t)=(y-b_1)\prod_{i=2}^{d_1}(a_iy-b_i)+t\ti{f}^{(d_1-1)}(y)+o(t),$$
where $f^{(d_1-1)}$ is the homogeneous part of $f$ of degree $d_1-1$. Here we work in $\C^2$ and omit the $z$ and $w$ variables as they do not occur in $f$. Let $$A=-\ti{f}^{(d_1-1)}(b_1)/\prod_{i=2}^{d_1}(a_ib_1-b_i).$$ We claim that for a suitable series in $o(t)$ the function $\ti{f}$ vanishes on the curve parametrized for small $t$ by $(y(t),t)$ for $y(t)=b_1+At+o(t)$. Indeed, we have $\ti{f}(y(t),t)=o(t)$ and we may continue the process of defining $y(t)$ up to $o(t^k)$ in such a manner that $\ti{f}(y(t),t)=o(t^{k+1})$, passing with $k$ to infinity we obtain that $\ti{f}(y(t),t)=0$. Namely, if $y_k(t)=b_1+At+\sum_{j=2}^kA_jt^j$ and $\ti{f}(y_k(t),t)=o(t^k)=A'_{k+1}t^{k+1}+o(t^{k+1})$ then setting $A_{k+1}=-A'_{k+1}/\prod_{i=2}^{d_1}(a_ib_1-b_i)$ we obtain $\ti{f}(y_{k+1}(t),t)=o(t^{k+1})$.

Now we will construct $z_p(t)$ and $w_p(t)$, for $1\leq p\leq d_2(d_2-1)$, so that
$$b_p(t)=(y(t),z_p(t),w_p(t),t)$$
parametrizes for small $t$ a curve in the zero locus of $(\ti{f},\ti{g},\ti{h})$. Take $z_p(t)=B_pt+o(t)$ and $w_p(t)=C_p+o(t)$. We have $$\ti{g}(b_p(t))=g(B_p,C_p)t^{d_2}+o(t^{d_2}),$$
thus we need to ensure that $g(B_p,C_p)=0$. Furthermore,
$$\ti{h}(b_p)=$$
$$g_z(B_p,C_p)t^{d_2-1}(b\ov{f}_x+d\ov{f}_y)(1,b_1+o(1),t)
-g_w(B_p,C_p)t^{d_2-1}(a\ov{f}_x+c\ov{f}_y)(1,b_1+o(1),t)=$$
$$[g_z(B_p,C_p)(bf^{(d_1)}_x+df^{(d_1)}_y)(1,b_1)
-g_w(B_p,C_p)(af^{(d_1)}_x+cf^{(d_1)}_y)(1,b_1)]t^{d_2-1}+o(t^{d_2-1}),$$
thus we need to ensure that $(D_1g_z-D_2g_w)(B_p,C_p)=0$, where $D_1$ and $D_2$ are constants dependent on $a,b,c,d$ and the homogeneous part of $f$ of degree $d$. So we have to take $B_p$ and $C_p$ such that the point $(B_p,C_p)$ lies in the intersection of curves given in $\C^2$ by $g=0$ and $D_1g_z-D_2g_w=0$. Moreover, we are able to expand $z_p(t)$ and $w_p(t)$ up to degree $k$ in such manner that $\ti{g}b_p(t)=o(t^{d_2+k-1})$ and $\ti{g}(b_p(t))=o(t^{d_2+k-2})$. Indeed, if we expand $z_p(t)$ and $w_p(t)$ up to degree $k$ then we can obtain the coefficients of $t^{k+1}$ by solving a system of two linear equations with two variables. The determinant of the matrix associated with the system is $$[g_z(D_1g_{zw}-D_2g_{ww})-g_w(D_1g_{zz}-D_2g_{zw})](B_p,C_p),$$
which is nonzero for generic $g$. Thus the system of equations has a unique solution.

To make sure that we indeed obtain $d_2(d_2-1)$ distinct curve germs $b_p(t)$ we need to use the genericity of $X$, $Y$ and $a,b,c,d$ to ensure that $g=D_1g_z-D_2g_w=0$ has $d_2(d_2-1)$ distinct solutions. To abbreviate the argument we will say in this case that $(f,g,D_1,D_2)$ is good. For any fixed pair $(D_1,D_2)\in\C^2\setminus\{(0,0)\}$ there is a dense subset of $\Omega_2(d_1)\times \Omega_2(d_2)$, the product of spaces of bivariate polynomials of degree, respectively, at most $d_1$ or $d_2$, such that $(f,g,D_1,D_2)$ is good. This implies, that the set of good $(f,g,D_1,D_2)$ is dense in $\Omega_2(d_1)\times \Omega_2(d_2)\times\C^2$. It is also constructible, so it contains an open dense subset. In particular, for generic $f$ and $g$ there is an open dense subset $V\subset\C^2$ such that for all $(D_1,D_2)\in V$ the quadruple $(f,g,D_1,D_2)$ is good. The values $f^{(d)}_x(1,b_1)$ and $f^{(d)}_y(1,b_1)$ depend only on $X$ and for generic $X$ they are nonzero, i.e., $f=0$ does not intersect $f_x=0$ nor $f_y=0$ at infinity. Thus for generic $a,b,c,d$ the pair $((bf^{(d)}_x+df^{(d)}_y)(1,b_1), (af^{(d)}_x+cf^{(d)}_y)(1,b_1))$ is in $V$.

Summarizing, $\ov{C}$ has $d_1+d_2$ points of at infinity: $P_i$ for $1\leq i\leq d_1$ and $Q_i$ for $1\leq i\leq d_2$. At each of $P_i$ the curve $\ov{C}$ has $d_2(d_2-1)$ distinct branches. By symmetry at each of $Q_i$ the curve $\ov{C}$ has $d_1(d_1-1)$ distinct branches. Thus $\ov{C}$ has $d_1d_2(d_1+d_2-2)$ branches at infinity. In particular $\ov{C}$ is a scheme-theoretic complete intersection of $\ov{f}$, $\ov{g}$ and $\ov{h}$. Moreover, $C$ is a curve of degree $d_1d_2(d_1+d_2-2)$ and, consequently, $C'$ is a curve of degree $d_1d_2(d_1+d_2-2)$ as well.

{\bf Step 3:} Computing the number of cusps of $\GSDS(X,Y)$.

Note that each cusp of $\GSDS(X,Y)$ is the image of a unique critical point of $\pi|_C$. Thus $c$, the number of cusps, is equal to the number of points of $C$ at which the matrix
\begin{equation}
M=\left[\begin{matrix} f_x & f_y & 0 &0\\ 0 & 0 & g_z & g_w\\h_x & h_y & h_z &h_w \\ 1 & 0 & a & b\\0 & 1 & c & d  \\   \end{matrix}\right].
\end{equation}
fails to have maximal rank. Since $g$, $g_z$ and $g_w$ don't have common zeroes, the second row of $M$ does not vanish on $C$. Moreover, $h$ vanishes on $C$, so the first, second, fourth and fifth row of $M$ are linearly dependent. It follows that the rank of $M$ does not decrease after removing the first row. We denote the determinant of $M$ without the first row by $s$, i.e.,
$$s=h_x(ag_w-bg_z)+h_y(cg_w-dg_z)-h_zg_w+h_wg_z.$$

Now we compute the intersection multiplicity of the zero locus of $\ov{s}$ with $\ov{C}$ at $P_i$ and $Q_i$. For $P_i$ it suffices to determine the order of $\ov{s}(b_p(t))$ for the branches $b_p(t)$ of $\ov{C}$ at $P_1$. Recall that $$b_p(t)=(1:b_1+o(1):B_pt+o(t):C_pt+o(t):t),$$
thus $h_x$, $h_y$, $g_w$ and $g_z$ composed with $b_p(t)$ have order $d_2-1$ and $h_z$ and $h_w$ composed with $b_p(t)$ have order $d_2-2$. By genericity of $f$ and $g$ the initial coefficient does not vanish, so the order of $\ov{s}(b_p(t))$ is $2d_2-3$. The computation will not be fully symmetric for $Q_i$. A branch at a point $Q_i$ will have the form
$$b_q(t)=(A_qt+o(t):B_qt+o(t):c_1+o(1):d_1+o(1):t).$$
Thus, after composing with $b_q(t)$, $g_w$ and $g_z$ will have order $0$, $h_x$ and $h_y$ will have order $d_1-2$ and $h_z$ and $h_w$ will have order $d_1-1$. So the order of $\ov{s}(b_q(t))$ is $d_2-2$. Summing up we obtain that the intersection multiplicity is $d_2(d_2-1)(2d_2-3)$ at points $P_i$ and $d_1(d_1-1)(d_1-2)$ at points $Q_i$.

Note that $C$ and $\{s=0\}$ intersect transversally because otherwise $\GSDS(X,Y)$ would have a singularity other than a cusp or a node. Thus by Bezout's Theorem the number of intersection points is
$$c=d_1d_2(d_1+d_2-2)(d_1+2d_2-4)-d_1d_2(d_2-1)(2d_2-3)-d_2d_1(d_1-1)(d_1-2)=$$
$$=3d_1d_2(d_1-1)(d_2-1).$$

{\bf Step 4:} Computing the Euler characteristic and genus.

Note that the mapping $\pi\colon X\times Y\to\C^2$ is a ramified covering of degree $d_1d_2$, with the discriminant $C'$. Let $n$ denote the number of nodes of $C'$ and $c$ denote the number of cusps. Let us recall that the mapping $p=\pi_{|C}: C\to C'$ is generically one-to-one and the fiber $p^{-1}(a)$ has more than one point (in fact then it has exactly two points) only if $a$ is a node. Consequently we can write the following equality:
$$\chi(X\times Y)=
d_1d_2(1-\chi(C'))+(d_1d_2-1)(\chi(C')-n-c)+(d_1d_2-2)(n+c).$$

The equation simplifies to:
\begin{equation}\label{eq_Euler}
\chi(C')+n+c=d_1d_2-\chi(X)\chi(Y).
\end{equation}

Moreover we have $\chi(X)+d_1=2-2g(X)$, so $\chi(X)=-d_1(d_1-2)$. Similarly, $\chi(Y)=-d_2(d_2-2)$. Furthermore, $\chi(C')-n=\chi(C)-2n$, so $\chi(C)=\chi(C')+n$. Substituting these equalities to equation (\ref{eq_Euler}) we obtain:
$$\chi(C)=d_1d_2-d_1d_2(d_1-2)(d_2-2)-c=-d_1d_2(4d_1d_2-5(d_1+d_2)+6).$$

From the equality $\chi(C)+d_1d_2(d_1+d_2-2)=2-2g(C)$ we obtain:
$$g(C)=d_1d_2(2d_1d_2-3(d_1+d_2)+4)+1.$$

{\bf Step 5:} Computing the number of nodes of $\GSDS(X,Y)$.

We will use the following theorem of Serre (see \cite{mil}, p. 85):

\begin{theorem}\label{thmgenusdelta}
If $\Gamma$ is an irreducible curve of degree $d$ and genus $g$  in the complex projective plane
then $$\frac{1}{2}(d-1)(d-2)= g + \sum_{z\in \Sing(\Gamma)} \delta_z,$$
where $\delta_z$ denotes the delta invariant of a point $z$.
\end{theorem}

In order to use this Theorem we will need to compute the delta invariants of points at infinity of $C'$. Note that if $b$ is a branch of $\ov{C}$ at $P_i$, then $\pi(b)$ will be a branch of $\ov{C'}$ at $P_i'=(a_i:b_i:0)$. If $b$ is a branch of $\ov{C}$ at $Q_i$ then $\pi(b)$ will be a branch of $\ov{C'}$ at $Q_i'=(ac_i+bd_i:cc_i+dd_i:0)$. Thus for generic projection $\pi$, i.e. generic quadruple $(a,b,c,d)$, the curves $\ov{C}$
and $\ov{C'}$ will have the same number of distinct points at infinity.

Now we will show, that the branches at infinity of $\ov{C'}$ are pairwise transversal. By symmetry it is enough to show this for branches at $P_1'$. Consider a branch $b_p(t)$ for $t$ small but nonzero. In the projective space $\PP^4$ we have
$$b_p(t)=(1:b_1+At+o(t):B_pt+o(t):C_pt+o(t):t),$$
thus in the space $\C^4$ in which $C$ was originally defined we have $$b_p(t)=(t^{-1},b_1t^{-1}+A+o(1),B_p+o(1),C_p+o(1)).$$
Consequently,
$$\pi(b_p(t))=(t^{-1}+aB_p+bC_p+o(1),b_1t^{-1}+A+cB_p+dC_p+o(1)).$$
In the projective space $\PP^2$ we have
$$\pi(b_p(t))=(t^{-1}+aB_p+bC_p+o(1):b_1t^{-1}+A+cB_p+dC_p+o(1):1)=$$
$$=(1+(aB_p+bC_p)t+o(t):b_1+(A+cB_p+dC_p)t+o(t):t)=$$
$$=(1:b_1+(A+(c-b_1a)B_p+(d-b_1b)C_p)t+o(t):t+o(t)).$$

Thus a branch $b_p'(t)$ of $\ov{C'}$ at $P_1'$ is parametrized by $(b_1+(A+(c-b_1a)B_p+(d-b_1b)C_p)t+o(t),t+o(t))$, so $[A+(c-b_1a)B_p+(d-b_1b)C_p,1]$ is
a tangent vector at $P_1'$. Since the points $(B_p,C_p)$ are distinct we conclude that for a generic quadruple $(a,b,c,d)$ we have $(c-b_1a)(B_{p_1}-B_{p_2})+(d-b_1b)(C_{p_1}-C_{p_2})\neq 0$ for $p_1\neq p_2$. So distinct branches have distinct tangent spaces.

Since $\ov{C'}$ has $d_2(d_2-1)$ pairwise transversal branches at $P_i'$ we obtain $\delta_{P_i'}{\ov{C'}}= d_2(d_2-1)[d_2(d_2-1)-1]/2$. Similarly, $\delta_{Q_i'}{\ov{C'}}= d_1(d_1-1)[d_1(d_1-1)-1]/2$.

By the Serre formula we have  
$$\frac{1}{2}\deg(C')(\deg(C')-1)=g(C')+n+c+d_1\delta_{P_i'}{\ov{C'}}+d_2\delta_{Q_i'}{\ov{C'}}.$$
The only unknown value in this equation is the number of nodes. After simplification we obtain
$$n=2{d_1 \choose 2}{d_2 \choose 2}[(d_1+d_2)^2-d_1-d_2-10].$$

\end{proof}

Taking $Y=X$ in Theorem \ref{th_GSDSXY} we immediately obtain:

\begin{corollary}\label{cor_GSDSX}
Let $X\subset \C^2$ be a generic plane curve of degree $d\geq 2$. Then $\GSDS(X)$ is an irreducible curve of degree $2d^2(d-1)$ and genus $g=2d^2(d^2-3d+2)+1$ with $c=12{d \choose 2}^2$
cusps and $n=4{d \choose 2}^2[2d^2-d-5]$ nodes. Moreover, $\chi(\GSDS(X))=-d^2(4d^2-10d+6).$
\end{corollary}

In particular for $d=2$ we have:

\begin{corollary}\label{koloex}
Let $X=\{ (x,y)\in \C^2: x^2+y^2=1\}.$ Then $GSDS(X)$ is an irreducible elliptic curve with  $12$ cusps and $4$ nodes. 
\end{corollary}

Note that neither $X=\{ (x,y)\in \C^2: x^2+y^2=1\}$ nor $Y=\{ (x,y)\in \C^2: xy=1\}$ satisfy the conditions of genericity imposed in the proof of Theorem \ref{th_GSDSXY}, however a generic curve of degree $2$ can be reduced to $X$ and $Y$ via affine transformations.

In the real case we have a finite number of possible $\GSDS(X)$ and $\GSDS(Y)$ for the real circle $X$ and the real hyperbola $Y$. As an example we present in Figure~\ref{pict} the images of $\GSDS(X)$ and $\GSDS(Y)$ for $G(x,y)=(x,y)$ and $H(x,y)=(1.1x+0.1y,-0.2x+0.9y)$.
\begin{center}
\begin{figure}[ht]
\includegraphics[scale=0.4]{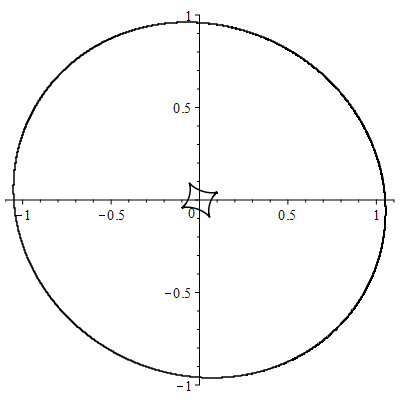}
\includegraphics[scale=0.4]{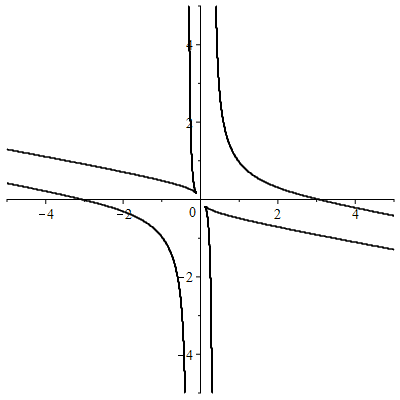}
\caption{ $\GSDS(X)$ and $\GSDS(Y)$.}\label{pict}
\end{figure}
\end{center}

The real curve $\GSDS(X)$ has four cusps and the real curve $\GSDS(Y)$ has two cusps.

\vspace{5mm}

{\bf Acknowledgment.}
All authors were partially supported by the grant of Narodowe Centrum Nauki number 2019/33/B/ST1/00755. The authors are grateful to professor M.A.S. Ruas for many helpful discussions.


\vspace{10mm}



\begin{thebibliography}{99}

\bibitem{aki} 	Y. Akizuki, {\em Theorems of Bertini on linear systems}, J. Math. Soc. Japan {\bf 3}, (1951), 170--180.

\bibitem{berry} M. Berry, {\em Semi-classical mechanics in
     phase space: A study of Wigner's function}, {Phil. Trans. Roy.
     Soc.}  {\bf 287} (1977), 237--271.


\bibitem{Bruce} J.W. Bruce, P.J. Giblin, C.G. Gibson, {\em Symmetry
sets}, {Proc. Roy. Soc. Edinburgh} {\bf 101A}, (1983), 163--186. 

\bibitem{Damon} J. Damon, {\em Smoothness and geometry of boundaries associated
to skeletal structures I: Sufficient conditions for smoothness},
{Ann. Inst. Fourier}, Grenoble, {\bf 53}, 6 (2003), 1941--1985. 

\bibitem{Damon1} J. Damon, {\em  Determining the Geometry of
Boundaries of Objects from Medial Data},  {International Journal of
Computer Vision} {\bf 63}(1), (2005), 45--64. 

\bibitem{DP}  W. Domitrz, P. Rios, {\em Singularities of equidistants and global centre symmetry sets of Lagrangian 
submanifolds}, Geom. Dedicata {\bf 169}, (2014), 361--382.

    

   

\bibitem{fjr} M. Farnik, Z. Jelonek, M.A.S. Ruas, {\em  Whitney theorem for complex polynomial mappings}, Math. Z. {\bf 295} (2020), 1039--1065.

\bibitem{GH}P.J.Giblin, P.A.Holtom, {\em   The centre symmetry set}, Geometry and Topology of Caustics,
{Banach Center Publications} Vol. {\bf 50,} ed. S. Janeczko and
V.M. Zakalyukin, Warsaw, 1999, 91--105. 

\bibitem{GJ}P.J. Giblin, S. Janeczko, {\em  Geometry of curves and surfaces through the contact map},
{Topol. Appl.} {\bf 159} (2012), 379--380. 

\bibitem{har} R. Hartshorne, {\em Algebraic
Geometry}, Springer Verlag, New York, 1987.



\bibitem{Jan} S. Janeczko, {\em Bifurcations of the Center of Symmetry}, {Geom. Dedicata} {\bf 60},
(1996), 9--16. 

\bibitem{jjr} S. Janeczko, Z. Jelonek, M.A.S. Ruas, {\em Symmetry defect of algebraic varieties}, {\em Asian J. Math.} {\bf 18}, (2014), 525--544.

\bibitem{jel} Z. Jelonek, {\em
On semi-equivalence of generically-finite polynomial mappings}, Math. Z. {\bf 283}, (2016), 133--142.




\bibitem{kl1} S. Kleiman, {\em The enumerative theory of singularities},
	 Real and complex singularities -- Proc. Ninth Nordic Summer School NAVF Sympos. Math., Oslo (1976), 297--396. 
 
	 
\bibitem{math}J.N.  Mather, {\em Generic projections},  Ann. of Math. (2) {\bf 98}, (1973), 226--245.

\bibitem{math2}J.N.  Mather, {\em On Thom-Boardman singularities}, Dynamical Systems Proceedings of a Symposium Held at the University of Bahia, Salvador, Brasil, July 26-August 14, 1971, (1973),  233--248.
	 
\bibitem{mil} J. Milnor, {\em  Singular points of complex hypersurfaces}, Annals of Mathematics Studies, Princeton University Press,  (1968).

           
              








\end{thebibliography}
\end{document}